\documentclass{article}

\usepackage{natbib}

\usepackage{amssymb}

\usepackage[english]{babel}
\usepackage{textcomp}
\usepackage{amsmath}
\usepackage{amsthm}

\newtheorem{definition}{Definition}[section]
\newtheorem{lemma}[definition]{Lemma}
\newtheorem{remark}[definition]{Remark}

\newtheorem{theor}[definition]{Theorem}
\newtheorem{corollary}[definition]{Corollary}

\newcommand{\E}{\mathbb{E}}   
\newcommand{\Pro}{\mathbb{P}} 
\newcommand{\Var}{\text{Var }}

\newcommand{\N}{\mathbb{N}}   
\newcommand{\Q}{\mathbb{Q}}   
\newcommand{\R}{\mathbb{R}}   

\begin{document}



\begin{center}
\Huge
{Exact Edgeworth expansion for a L\'{e}vy process}
\end{center}
\thispagestyle{empty}
\normalsize
{Heikki J. Tikanm\"{a}ki, Helsinki University of Technology, Institute of Mathematics, P.O.Box 1100, FI-02015 TKK, Finland. E-mail: heikki.tikanmaki@tkk.fi}

\begin{abstract}
The one dimensional distribution of a L\'{e}vy process is not known in general even though its characteristic function is given by the famous L\'{e}vy-Khinchine theorem. This article gives an exact series representation for the one dimensional distribution of a L\'{e}vy process satisfying certain moment conditions. Moreover, this work clarifies an old result by Cram\'{e}r on Edgeworth expansions for the distribution function of a L\'{e}vy process.

Keywords: Asymptotic expansions, Cram\'{e}r's condition, cumulants, Edgeworth approximation, L\'{e}vy process

AMS subject classification: 60G51, 60E07, 60G50.
\end{abstract}

\newpage
\section{Introduction}
\label{intro}
The L\'{e}vy-Khinchine theorem gives the characteristic function of a L\'{e}vy process. In spite of this, the distribution of a L\'{e}vy process is not analytically known, except in few special cases such as the Brownian motion, the Poisson process and the gamma process. For example, the distribution function of the compound Poisson process is not known in general despite its popularity as a risk process in insurance applications.

This article has two contributions. First of all, this article introduces some sufficient extra conditions to get an exact Edgeworth type series representation for the one dimensional distribution of a L\'{e}vy process in the presence of all moments. Secondly, this paper goes beyond an old result on Edgeworth approximation introduced by~\citet{cramer} as an analogue to the i.i.d. sum case. This article clarifies the connection between the distribution functions of L\'{e}vy processes and classical approximation results of sums of independent random variables. As a consequence, we will give an approximation method for the distribution of spectrally positive (negative) L\'{e}vy processes. This kind of processes are widely used in modern insurance models, see e.g.~\citet{kluppelberg}.

Beside the insurance applications, the results of this article could be applicable in the simulations of L\'{e}vy processes. In fact, the classical Edgeworth approximation has been used for getting error estimates for simulations of the small jumps of a L\'{e}vy process in~\citet{asmussenrosinski}. Moreover, the exact series representation might be useful tool for the study of theoretical properties of L\'{e}vy processes.

There are lots of approximation results in the literature. The normal approximation approximates well asymptotically the distribution function of a L\'{e}vy process when $t \rightarrow \infty$ if the third moment exists, see for instance \citet{valkeila}. Several authors have considered asymptotic expansions in the central limit theorem (Edgeworth approximation) for the sums of independent random variables to improve the normal approximation, see e.g. \citet{petrov1} or \citet{cramer}. These approximation methods are also well known in statistics and insurance mathematics (see~\citep{beard, kolassa}). Another approximation result (Theorem~\ref{cramerlevycase} here) is introduced for the distribution function of a L\'{e}vy process by~\citet{cramer} as an analogue to the i.i.d. sum result. This is the starting point of the research presented in this article.
\section{Definitions}
In this section, we define the concepts needed in the rest of the article.

Let us consider a probability space $(\Omega, \mathcal{F},\Pro)$. Let $X$ be a real valued random variable defined on this space. Let $v_X(s)=\E e^{isX}$ denote the characteristic function of $X$.
\label{definitions}
\begin{definition}[Cram\'{e}r's condition]
A random variable $X$ is said to satisfy Cram\'{e}r's condition if
\begin{equation*}
\limsup_{|s|\rightarrow \infty}\left | v_X(s)\right |<1.
\end{equation*}
\end{definition}
Remark~\ref{cramerremark} characterises Cram\'{e}r's condition in the case of L\'{e}vy processes.
\begin{definition}[Cumulants]
Let $k\in \N=\{1,2,\dots\}$. The cumulant of order $k$ of a random variable $X$ is defined as
\begin{equation*}
\gamma_{k}^{X}=\frac{1}{i^k}\left [ \frac{d^k}{ds^k}\log v_{X}(s)\right ]_{s=0}.
\end{equation*}
\end{definition}
Note that the cumulant of $X$ of order $k$ is finite if we have $\E |X|^k<\infty$.

We use the following definition for the (non-normalised) Hermite polynomial of order $n \in \N$
\begin{equation*}
H_n(x)=(-1)^n e^\frac{x^2}{2}\frac{d^n}{dx^n}e^{-\frac{x^2}{2}}.
\end{equation*}
This choice of the definition makes the series representation much simpler than the normalised one. The same choice is done e.g. by \citet{petrov1,kolassa}. With this definition one gets the identities
\begin{align*}
&H_{n+1}(x)=xH_n(x)-nH_{n-1}(x),\\
&H'_n(x)=nH_{n-1}(x)\quad \text{and}\\
&H_n(-x)=(-1)^nH_n(x)
\end{align*}
analogous to those in~\citet{nualart}.

We set $V_X^2=E X^2$. Let $\nu \in \N$ s.t. $\E |X|^{\nu+2}<\infty$. We are now ready to define the approximating function $Q_\nu^{X}$ to be used in the series approximations. We set
\begin{equation}
\label{Q}
Q_\nu^{X}(x)=-\frac{1}{\sqrt{2\pi}}e^{-\frac{x^2}{2}}\sum H_{\nu+2l-1}(x)\prod_{m=1}^\nu \frac{1}{k_m!}\left ( \frac{\gamma_{m+2}^X}{(m+2)!V_X^{m+2}}\right)^{k_m},
\end{equation}
where the summation is extended over the non-negative integer solutions $(k_1,\dots,k_\nu)$ of the equation $k_1+2k_2+\dots+\nu k_\nu=\nu$. Here we have $l=\sum_{j=1}^\nu k_j$. The first few of these functions are
\begin{align*}
Q_1^{X}(x)=&-\frac{1}{\sqrt{2\pi}}e^{-\frac{x^2}{2}}(x^2-1)\frac{\gamma_{3}^X}{6 V_X^3},\\
Q_2^{X}(x)=&-\frac{1}{\sqrt{2\pi}}e^{-\frac{x^2}{2}}((x^5-10x^3+15x)\frac{(\gamma_3^X)^2}{72 V_X^6}+(x^3-3x)\frac{\gamma_4^X}{24V_X^4}),\\
Q_3^{X}(x)=&-\frac{1}{\sqrt{2\pi}}e^{-\frac{x^2}{2}}((x^8-28x^6+210x^4-420x^2+105)\frac{(\gamma_3^X)^3}{1296V_X^9}+\\&(x^6-15x^4+45x^2-15)\frac{\gamma_3^X \gamma_4^X}{144V_X^7}+(x^4-6x^2+3)\frac{\gamma_5^X}{120V_X^5}).
\end{align*}
The approximating function of order zero is the cumulative distribution function of the standard normal distribution $\Phi(x)$.

In the remaining of this article the process $X=(X_t)_{t\geq 0}$ is assumed to be a L\'{e}vy process on $\R$. The standard definition for L\'{e}vy processes can be found for instance from~\citet{bertoin} or \citet{kyprianou}. The approximation results are written for centered processes i.e. $\E X_1=0$.

We use the following version of the L\'{e}vy-Khinchine theorem to represent the characteristic function $v_{X_t}(s)$. The theorem can be found in one form or another for example in \citet{bertoin,cont,sato}.
\begin{theor}[L\'{e}vy-Khinchine]
There are unique $\sigma^2\geq 0$, $\rho\in \R$ and a Radon measure $\mu$ on $\R \backslash \{0\}$ satisfying
\begin{equation*}
\int_{\R \backslash \{0\}}\min{(u^2,1)}d\mu(u)<\infty
\end{equation*}
 such that
\begin{equation*}
\psi(s)=-\frac{1}{2}\sigma^2 s^2+i\rho s+\int_{\R \backslash \{0\}}(e^{isu}-1-isu1_{\{|u|\leq 1\}})\mu(du)
\end{equation*}
and
\begin{equation*}
v_{X_t}(s)=e^{t\psi(s)}.
\end{equation*}
\end{theor}
The measure $\mu$ is called the L\'{e}vy measure of $X$ and $(\sigma^2,\rho,\mu)$ is the characteristic triplet of $X$.
\begin{remark}
\label{cramerremark}
The random variable $X_t$ satisfies Cram\'{e}r's condition if the law of $X_t$ has absolutely continuous component w.r.t. Lebesgue measure. This follows e.g. if $\sigma^2 \neq 0$ or $\mu$ has absolutely continuous component w.r.t. Lebesgue measure.

Moreover, if $X_t$ satisfies Cram\'{e}r's condition for some $t>0$, then $X_t$ satisfies the same condition for all $t>0$ by the L\'{e}vy-Khinchine theorem.
\end{remark}
\section{Approximation results}
\label{results}
In the literature, there are lots of classical asymptotic expansion results for the i.i.d. sum case. I.i.d. sums are in some sense the discrete time analogues of L\'{e}vy processes. The following theorem is presented in~\citet{petrov1}.
\begin{theor}
\label{iid-case}
Let $\{Y_j\}_{j=1}^n$ be a sequence of i.i.d. random variables satisfying Cram\'{e}r's condition s.t. $\E Y_1=0$ and $\E |Y_1|^k<\infty$ for some integer $k\geq 3$. Then
\begin{equation*}
\Pro\left ( \sum_{j=1}^n Y_j < \sqrt{n}V_{Y_1}x\right )=\Phi(x)+\sum_{\nu=1}^{k-2}Q_\nu^{Y_1}(x)n^{-\frac{\nu}{2}}+o\left(n^{-\frac{k-2}{2}}\right)
\end{equation*}
uniformly in $x\in \R$.
\end{theor}
This kind of results are presented also in~\citet{petrov2,kolassa,cramer}. Generalisation of Theorem~\ref{iid-case} is presented by~\citet{cramer} as an analogue to corresponding i.i.d. sum result:
\begin{theor}
\label{cramerlevycase}
Let $X_1$ satisfy Cram\'{e}r's condition, $\E X_1=0$ and $k\geq 3$ be such an integer that $\E |X_1|^k<\infty$. Then 
\begin{equation*}
\Pro(X_t<xV_{X_t})=\Phi(x)+\sum_{\nu=1}^{k-3}Q^{X_1}_\nu(x)t^{-\frac{\nu}{2}}+O\left(t^{-\frac{k-2}{2}}\right).
\end{equation*}
\end{theor}
In fact, \citet{cramer} introduces the form for the functions $Q_\nu^{X_1}(x)$ only implicitly. See \citet{cramer} pages 72, 98 and 99.

Next we are going to present some lemmata to scale the approximating functions $Q_\nu^{X_t}(x)$ with respect to $t$. The first of them is well-known but it is included here for convenience.
\begin{lemma}
Let $k \in \N$ be s.t. $\E |X_1|^k<\infty$. Then
\begin{equation*}
\gamma_k^{X_t}=t\gamma_k^{X_1}.
\end{equation*}
\end{lemma}
\begin{proof}
Take $q\in \Q^+$. Now $q=\frac{m}{n}$ for some $m,n\in\N$ and
\begin{equation*}
\gamma_k^{X_{\frac{1}{n}}}=\frac{1}{i^k}\left [ \frac{d^k}{ds^k} \log v_{X_1}(s)^\frac{1}{n}\right ]_{s=0}=\frac{1}{n}\frac{1}{i^k}\left [ \frac{d^k}{ds^k}\log v_{X_1}(s)\right]_{s=0}=\frac{1}{n}\gamma_k^{X_1}.
\end{equation*}
By repeating the previous argument we get
\begin{equation*}
\gamma_k^{X_q}=m\gamma_k^{X_{\frac{1}{n}}}=\frac{m}{n}\gamma_k^{X_1}=q\gamma_k^{X_1}.
\end{equation*}
The general claim follows now by a simple density argument.
\end{proof}
\begin{lemma}
\label{scaling}
Let $\nu \in \N$ be s.t. $\E |X_1|^{\nu+2}<\infty$, then
\begin{equation*}
Q_\nu^{X_t}(x)=t^{-\frac{\nu}{2}}Q_\nu^{X_1}(x), \quad \text{for }x\in\R.
\end{equation*}
\end{lemma}
\begin{proof}
By definition,
\begin{equation*}
Q_\nu^{X_t}(x)=f(x)\sum H_{\nu+2l-1}(x)\prod_{m=1}^\nu \frac{1}{k_m!}\left ( \frac{\gamma_{m+2}^{X_t}}{(m+2)!V_{X_t}^{m+2}}\right)^{k_m},
\end{equation*}
where $f(x)=-\frac{e^{-\frac{x^2}{2}}}{\sqrt{2\pi}}$ and the summation is extended over all non-negative integer solutions of the equation $\sum_{j=1}^\nu jk_j=\nu$, and we have $l=\sum_{j=1}^\nu k_j$.
\begin{align*}
Q_\nu^{X_t}(x)&=f(x)\sum H_{\nu+2l-1}(x)\prod_{m=1}^\nu \left ( \frac{t\gamma^{X_1}_{m+2}}{(m+2)!(\sqrt{t}V_{X_1})^{m+2}}\right )^{k_m}\\ &=f(x)\sum H_{\nu+2l-1}(x)\left (\prod_{m=1}^\nu t^{-\frac{1}{2}m k_m} \right )\cdot \left ( \prod_{m=1}^\nu \frac{1}{k_m!}\left ( \frac{\gamma_{m+2}^{X_1}}{(m+2)!V_{X_1}^{m+2}}\right)^{k_m}\right)\\ &=f(x)\sum t^{-\frac{1}{2}\sum_{m=1}^\nu mk_m}H_{\nu+2l-1}(x)\prod_{m=1}^\nu \left ( \frac{\gamma_{m+2}^{X_1}}{(m+2)!V_{X_1}^{m+2}}\right )^{k_m}\\=& t^{-\frac{\nu}{2}}Q_\nu^{X_1}(x).
\end{align*}
In the last step, we used the fact that $\nu=k_1+2k_2+\dots+\nu k_\nu$.
\end{proof}
We get the following result by combining the classical results to the previous lemmata and using some continuity arguments.
\begin{corollary}
\label{corolla}
Let $k\geq 3$ be integer s.t. $\E |X_1|^k<\infty$ and let $X_1$ satisfy Cram\'{e}r's condition and $\E X_1=0$. Then 
\begin{align*}
\Pro(X_t<xV_{X_t})&=\Phi(x)+\sum_{\nu=1}^{k-2}Q_\nu^{X_1}(x)t^{-\frac{\nu}{2}}+o\left ( t^{-\frac{k-2}{2}}\right)\\&=\Phi(x)+\sum_{\nu=1}^{k-2}Q_\nu^{X_t}(x)+o\left (t^{-\frac{k-2}{2}}\right), \quad \text{uniformly in }x \in \R.
\end{align*}
\end{corollary}
From now on in this paper, we assume (if not otherwise stated) that $\E X_1=0$, $X_1$ satisfies Cram\'{e}r's condition and has moments of all orders i.e.
\begin{equation*}
\E |X_1|^\nu<\infty, \quad \text{for }\nu \in \N.
\end{equation*}
Now we have everything ready for introducing the main results of the article to get exact series representations. The proofs are in Section~\ref{proofs}. In the following Theorems~\ref{main},~\ref{main2} and~\ref{main3}, $\mu$ is assumed to be the L\'{e}vy measure of process $X$.
\begin{theor}
\label{main}
Let the L\'{e}vy measure of $X$ have bounded support, then we get for $x_1<x_2$ points of continuity of $\Pro(X_t<\cdot V_{X_t})$ that
\begin{align*}
\label{representation}
&\Pro\left (x_1 <\frac{X_t}{V_{X_t}}< x_2\right)=\Pro(X_t<x_2V_{X_t})-\Pro(X_t<x_1V_{X_t})\\=&\Phi(x_2)-\Phi(x_1)+\sum_{\nu=1}^\infty \left (Q_\nu^{X_t}(x_2)-Q_\nu^{X_t}(x_1) \right)\\=&\Phi(x_2)-\Phi(x_1)+\sum_{\nu=1}^\infty\left (  Q_\nu^{X_1}(x_2)-Q_\nu^{X_1}(x_1)\right )t^{-\frac{\nu}{2}}.
\end{align*}
\end{theor}
There is some discussion about the L\'{e}vy measures with bounded support for example in~\citet{sato}. In fact, this is a reasonable class to be considered in the simulations because of the practical limitations.

Nevertheless, the result of Theorem~\ref{main} is true with more general conditions:
\begin{theor}
\label{main2}
Let $\mu$ be s.t. for some $a\geq 0$, $\mu(x)1_{\{|x|>a\}}$ is absolutely continuous with respect to Lebesgue measure and for some $C,\epsilon>0$
\begin{equation*}
\frac{d\mu(x)}{dx}\leq C \exp \{-|x|^{1+\epsilon}\}, \quad \text{for }|x|\geq a.
\end{equation*}
Then the assertion of Theorem~\ref{main} holds.
\end{theor}
And even more generally we get the following:
\begin{theor}
\label{main3}
Assume that there are $a\geq 0$ and $C,\epsilon>0$ s.t.
\begin{equation*}
\mu((-x-1,-x],[x,x+1))\leq C \exp\{-x^{1+\epsilon}\}, \quad \text{for } x\geq a.
\end{equation*}
Then the representation of Theorem~\ref{main} holds.
\end{theor}

\begin{remark}
In the cases of Theorems~\ref{main},~\ref{main2}~and~\ref{main3}, we get some series representation also for other finite dimensional distributions since the series representation can be written for all increments separately.
\end{remark}
Moreover, we get a representation for the distribution function of the absolute value of a L\'{e}vy process as follows:
\begin{corollary}
Assume that the assumptions of~\ref{main},~\ref{main2} or~\ref{main3} hold. Then we get for $x> 0$ and $-x$ points of continuity of $\Pro(X_t<\cdot V_{X_t})$ that
\begin{equation}
\Pro(|X_t|< xV_{X_t})=2 \Phi(x)-1+2\sum_{\nu=1}^\infty Q_{2\nu}^{X_t}(x)=2\Phi(x)-1+2\sum_{\nu=1}^\infty Q_{2\nu}^{X_1}(x)t^{-\nu}
\label{pipe1}
\end{equation}
and
\begin{equation}
\Pro(|X_t|>xV_{X_t})=2-2\Phi(x)-2\sum_{\nu=1}^\infty Q_{2\nu}^{X_t}(x)=2-2\Phi(x)-2\sum_{\nu=1}^\infty Q_{2\nu}^{X_1}(x)t^{-\nu}.
\label{pipe2}
\end{equation}
\end{corollary}
\begin{proof}
\begin{align*}
&\Pro(|X_t|<xV_{X_t})=\Pro(X_t<xV_{X_t})-\Pro(X_t<-xV_{X_t})\\=&\Phi(x)-\Phi(-x)+\sum_{\nu=1}^\infty \left (Q_\nu^{X_t}(x)-Q_\nu^{X_t}(-x)\right)\\=&2\Phi(x)-1+\sum_{\nu=1}^\infty \left ( Q_\nu^{X_t}(x)-Q_\nu^{X_t}(-x)\right ).
\end{align*}
We use the symmetry condition for Hermite polynomials and get
\begin{align*}
&Q_\nu^{X_t}(x)-Q_\nu^{X_t}(-x)\\=&-\frac{e^{-\frac{x^2}{2}}}{\sqrt{2\pi}}\sum \left ( H_{\nu+2l-1}(x)-H_{\nu+2l-1}(-x)\right )\prod_{m=1}^\nu \left ( \frac{\gamma_{m+2}^{X_t}}{(m+2)!V_{X_t}^{m+2}}\right )^{k_m}\\=&2Q_\nu^{X_t}(x)1_{\{\nu=2p | p \in \N\}}.
\end{align*}
Equation~(\ref{pipe2}) is a direct consequence of~(\ref{pipe1}).
\end{proof}
If $X_t$ has density function for all $t>0$, we get the following:
\begin{corollary}
\label{densitycase}
Assume besides the assumptions of~\ref{main},~\ref{main2} or~\ref{main3} that $\frac{X_t}{V_{X_t}}$ has density function $g_{X_t}(s)$ for all $t>0$. Then
\begin{equation*}
g_{X_t}(x)=\frac{1}{\sqrt{2\pi}}e^{-\frac{x^2}{2}}+\sum_{\nu=1}^\infty \frac{d}{dx}Q_\nu^{X_t}(x).
\end{equation*}
\end{corollary}
 Corollary~\ref{densitycase} gives us together with the following lemma an exact series representation for the density function.
\begin{lemma}For $\nu \in \N$ we have
\begin{equation*}
\frac{d}{dx}Q_\nu^{X_t}(x)=\frac{1}{\sqrt{2\pi}}e^{-\frac{x^2}{2}}\sum H_{\nu+2l}(x)\prod_{m=1}^\nu \frac{1}{k_m!}\left ( \frac{\gamma^{X_t}_{m+2}}{(m+2)! V_{X_t}^{m+2}}\right )^{k_m},
\end{equation*}
with the notation of~(\ref{Q}).
\end{lemma}
\begin{proof}
\begin{align*}
\frac{d}{dx}Q_\nu^{X_t}(x)=&\left (\frac{d}{dx}\left ( -\frac{1}{\sqrt{2\pi}} e^{-\frac{x^2}{2}}\right )\right )\sum H_{\nu+2l-1}(x)\prod_{m=1}^\nu \frac{1}{k_m!}\left ( \frac{\gamma^{X_t}_{m+2}}{(m+2)! V_{X_t}^{m+2}}\right)^{k_m}\\&-\frac{1}{\sqrt{2\pi}}e^{-\frac{x^2}{2}}\sum \frac{d}{dx}H_{\nu+2l-1}(x)\prod_{m=1}^\nu \frac{1}{k_m!}\left ( \frac{\gamma^{X_t}_{m+2}}{(m+2)!V_{X_t}^{m+2}}\right)^{k_m}\\=&\frac{1}{\sqrt{2\pi}}e^{-\frac{x^2}{2}}\sum (x H_{\nu+2l-1}(x)-(\nu+2l-1)H_{\nu+2l-2}(x))\times\\ &\prod_{m=1}^\nu \frac{1}{k_m!}\left (\frac{\gamma^{X_t}_{m+2}}{(m+2)!V_{X_t}^{m+2}} \right )^{k_m}\\=&\frac{1}{\sqrt{2\pi}}e^{-\frac{x^2}{2}}\sum H_{\nu+2l}(x)\prod_{m=1}^\nu \frac{1}{k_m!}\left ( \frac{\gamma^{X_t}_{m+2}}{(m+2)! V_{X_t}^{m+2}}\right )^{k_m}.
\end{align*}
In the last step, we used the recursion formula for the Hermite polynomials.
\end{proof}

\begin{remark}
The approximation results of this section such as Theorem~\ref{main} give exact series representation also for any infinitely divisible distribution satisfying the conditions for $X_1$ in each theorem, since any infinitely divisible distribution can be considered as the one dimensional distribution of some L\'{e}vy process at time $1$.
\end{remark}
\section{Insurance Applications}
\label{applications}
Let us consider briefly L\'{e}vy processes with only positive (respectively negative) jumps and drift term. This is a reasonable class for risk processes, more precisely claim surplus processess in the sense of~\citet{asmussen}. This class includes spectrally positive (negative) L\'{e}vy processes without gaussian component in the sense of~\citet{kyprianou}.
\begin{remark}[Risk process]
Consider a L\'{e}vy process $X$ satisfying conditions of Theorem \ref{main}, \ref{main2} or \ref{main3}. Furthermore, assume that its L\'{e}vy measure is concentrated on positive reals and satisfies $\int_{\R \backslash \{0\}} |x|\mu(dx)<\infty$. Then there is some $x_1 \in \R$ s.t. $\Pro(X_t < x_1 V_{X_t})=0$ for all $t>0$ and thus we get easily a series representation for $\Pro(X_t < x_2 V_{X_t})$ alone.
\end{remark}
\begin{remark}
Throughout the paper we have assumed that the process is centered. This assumption is only technical since we can write for non-centered $X_t$ that
\begin{equation*}
\Pro(X_t<x)=\Pro(X_t-\E X_t<y V_{X_t-\E X_t }),
\end{equation*}
where
\begin{equation*}
y=\frac{x-\E X_t}{V_{X_t-\E X_t}}=\frac{x-\E X_t}{\sqrt{\Var{X_t}}}.
\end{equation*}
\end{remark}
The Edgeworth approximation is widely used in insurance applications \citep{beard,asmussen}. Present results justify the use of Edgeworth expansion of any order to approximate the claim surplus process in a L\'{e}vy driven model. By increasing the order of the approximation we will asymptotically get rid of the error term. The series representation can be written for all $t>0$. Naturally we will have to take more correction terms $Q_\nu^{X_t}(x)$ into account if $t$ is small or $|x|$ is large to get sharp estimates.

If there exist only $k$ first moments (the heavy tailed case), Corollary~\ref{corolla} tells to what extent one can refine the approximation.

Even the most restrictive case ot the main result of this article, Theorem~\ref{main} can be justified by actuarial reasoning. The bounded support of the L\'{e}vy measure corresponds to the case that the insurer has arranged an excess-of-loss reinsurance~\citep{asmussen}.
\section{Proofs}
\label{proofs}
The following lemma gives us a representation formula for the cumulants of a L\'{e}vy process. The result may be well known but it is included in this paper for convenience. It is worth mentioning that Cram\'{e}r's condition is not assumed in the following lemma. The condition~(\ref{intcond}) is used in the literature e.g. by~\citet{nualartschoutens}. This condition is enough to guarantee the existence of all moments. On the other hand, processes satisfying the assumptions of Theorem~\ref{main}, \ref{main2} or \ref{main3} also satisfy condition~(\ref{intcond}).
\begin{lemma}
\label{cumulantlemma}
Let $(\sigma^2,\rho,\mu)$ be the characteristic triplet of $X$. Furthermore, assume that for some $\lambda>0$ and for all $\delta>0$
\begin{equation}
\label{intcond}
\int_{\R \backslash{(-\delta,\delta)}} e^{\lambda |x|}\mu(dx)<\infty.
\end{equation}
Then
\begin{equation*}
\gamma_\nu^{X_1}=\int_{\R \backslash \{0\}} x^\nu \mu(dx),\quad \nu\geq 3.
\end{equation*}
and
\begin{equation*}
\gamma_2^{X_1}=\int_{\R \backslash \{0\}} x^2\mu(dx)+\sigma^2.
\end{equation*}
\end{lemma}
The proof is a straightforward computation using L\'{e}vy-It\^{o} decomposition and it is omitted. 
The next lemma gives us another characterisation of the condition on the L\'{e}vy measure in Theorem~\ref{main}. From now on in this article, we will use the following notation of scaled cumulants $\lambda_\nu^{X_t}=\frac{\gamma_\nu^{X_t}}{V_{X_t}^\nu}$, for $\nu \in \N$.
\begin{lemma}
\label{boundedsupport}
The L\'{e}vy measure of process $X$ is concentrated on some bounded interval is equivalent to the condition that there exists some $C>0$ s.t.
\begin{equation*}
\lambda_\nu^{X_1}\leq C^\nu, \quad \text{for all }\nu\in\N.
\end{equation*}
\end{lemma}
\begin{proof}
Let us first assume that such $C$ exists. Now we can use Lemma~\ref{cumulantlemma} and we get for $\nu\geq 3$ that
\begin{equation*}
\int_{\R\backslash \{0\}}x^\nu \mu(dx)\leq C^\nu V_{X_1}^\nu.
\end{equation*}
For even $\nu$, $|\gamma_\nu|=\gamma_\nu$. We know also by \citet{rudin} page 71 that it holds for $L^p(\mu)$ norms that
\begin{equation*}
||x||_{2n+1}\leq \max{(||x||_{2n},||x||_{2n+2})}, \quad \text{for } n\geq 1.
\end{equation*}
Hence there is some $D>0$ s.t. $\int_{\R \backslash \{0\}}|x|^\nu \mu(dx)\leq D^\nu$ for all $\nu \geq 4$. Moreover, we get
\begin{equation*}
D \geq ||x||_\nu\rightarrow ||x||_\infty \quad \text{as }\nu \rightarrow \infty.
\end{equation*}
Now $||\frac{x}{D}||_\infty\leq 1$ with respect to $\mu$. In other words, $\mu$ is concentrated on some bounded interval.

The other way is even simpler. Because $\mu$ is concentrated on some bounded interval, it follows that $||x||_\infty<\infty$. We can choose $C=\frac{1}{V_{X_1}}\sup_{\nu}||x||_\nu$.
\end{proof}
Now we have everything ready for the proofs of the main results.
\begin{proof}(Theorem~\ref{main})

Let us first work out the representation for the logarithm of the characteristic function i.e. the characteristic exponent  of the L\'{e}vy process.
\begin{align*}
&\sum_{\nu=2}^\infty \left|\frac{\lambda_\nu^{X_t}}{\nu!}(is)^\nu\right|=\sum_{\nu=2}^\infty \left | \frac{1}{\nu!}\frac{t \gamma_\nu^{X_1}}{t^{\frac{\nu}{2}}V_{X_1}^\nu}(is)^\nu\right |\\=&\sum_{\nu=2}^\infty \left | t^{-\frac{\nu-2}{2}}\frac{1}{\nu!}\frac{\gamma_\nu^{X_1}}{V_{X_1}^\nu}(is)^\nu\right |\\ =&\sum_{\nu=2}^\infty \left | t\frac{\lambda_\nu^{X_1}}{\nu!}\left ( \frac{is}{\sqrt{t}}\right )^\nu\right |\leq t \sum_{\nu=2}^\infty \frac{1}{\nu !}\left | \frac{Cs}{\sqrt{t}}\right |^\nu,
\end{align*}
which is bounded when $t>\epsilon>0$ and $|s|<K<\infty$ for arbitrary $\epsilon,K \in (0,\infty)$. In the last step, we used the characterisation of Lemma~\ref{boundedsupport}. Now this series is dominated by the series expansion of the exponential function and thus the series
\begin{equation*}
\sum_{\nu=2}^\infty \frac{\lambda_\nu^{X_t}}{\nu!}(is)^\nu
\end{equation*}
converges to an analytic function of $s$ when $t>0$ is fixed. Now, define
\begin{equation*}
f_{X_t}(s)=v_{X_t}\left(\frac{s}{V_{X_t}}\right).
\end{equation*}
By computing the cumulants, this notation gives for $n \in \N$
\begin{align*}
& \left [\frac{d^n}{ds^n}\log f_{X_t}(s)\right ]_{s=0}=\left [\frac{d^n}{ds^n} \log v_{X_1}\left ( \frac{s}{\sqrt{t}V_{X_1} }\right )^t\right ]_{s=0}\\=&t\left ( \frac{1}{\sqrt{t}V_{X_1}}\right )^n \left [\frac{d^n}{ds^n}\log v_{X_1}(s)\right]_{s=0}\\ =& t^{-\frac{n-2}{2}}i^n\frac{\gamma_n^{X_1}}{V_{X_1}^n}=t^{-\frac{n-2}{2}}i^n \lambda_n^{X_1}=i^n \lambda_n^{X_t}.
\end{align*}
Now
\begin{equation*}
\log f_{X_t}(s)=\sum_{\nu=2}^\infty \frac{\lambda_\nu^{X_1}}{\nu !}t^{-\frac{\nu-2}{2}}(is)^\nu.
\end{equation*}
We observe that $\lambda_2^{X_t}=1$ for all $t>0$. So we obtain
\begin{equation*}
f_{X_t}(s)=e^{-\frac{s^2}{2}}\exp \left ( \sum_{j=1}^\infty \frac{\lambda_{j+2}^{X_1}}{(j+2)!}t^{-\frac{j}{2}}(is)^{j+2}\right ).
\end{equation*}
Next, consider a more general form
\begin{equation*}
\exp{\left ( \sum_{j=1}^\infty \frac{\lambda^{X_1}_{j+2}}{(j+2)!}z^j u^{j+2}\right )}.
\end{equation*}
With fixed $u$, this series converges absolutely, uniformly in any compact set with respect to the parameter $z$. Thus in every compact set with respect to $z$, we rearrange the series of the exponential function and get a series representation with respect to $z$. Hence,
\begin{equation*}
\exp \left ( \sum_{j=1}^\infty \frac{\lambda_{j+2}^{X_1}}{(j+2)!}z^j u^{j+2}\right )= 1 + \sum_{\nu=1}^\infty P_\nu(u)z^\nu
\end{equation*}
for some polynomials $(P_\nu)_{\nu=1}^\infty$ that can be computed formally by compounding these two series, which is possible due to the absolute convergence. Now
\begin{equation*}
f_{X_t}(s)=e^{-\frac{s^2}{2}}+\sum_{\nu=1}^\infty P_\nu(is)e^{-\frac{s^2}{2}}t^{-\frac{\nu}{2}}.
\end{equation*}
By the inversion formula of the characteristic function \citep{petrov1}, we get for $x_1,x_2$ points of continuity of $\Pro(X_t<\cdot V_{X_t})$
\begin{align*}
&\Pro(X_t<x_2 V_{X_t})-\Pro(X_t<x_1 V_{X_t})\\ =& \frac{1}{2\pi}\lim_{T\rightarrow \infty}\int_{-T}^T \frac{e^{-isx_2}-e^{-isx_1}}{-is}\left ( e^{-\frac{s^2}{2}}+\sum_{\nu=1}^\infty P_\nu(is) e^{-\frac{s^2}{2}}t^{-\frac{\nu}{2}}\right )ds.
\end{align*}
With fixed $t>0$, the series inside the integral is absolutely convergent uniformly in compact sets with respect to $s$. Thus the integral is always well-defined and can be computed term-wise. Moreover, the limit exists since
\begin{align*}
&\int_{-\infty}^\infty \frac{e^{-isx_2}-e^{-isx_1}}{-is}f_{X_t}(s)ds\\&-\int_{-T}^T \frac{e^{-isx_2}-e^{-isx_1}}{-is}\left ( e^{-\frac{s^2}{2}}+\sum_{\nu=1}^\infty P_\nu(is)e^{-\frac{s^2}{2}}t^{-\frac{\nu}{2}}\right )ds\\=&\int_{|s|>T} \frac{e^{-isx_2}-e^{-isx_1}}{-is}f_{X_t}(s)ds\rightarrow 0, \quad \text{when } T\rightarrow \infty,
\end{align*}
since $f_{X_t}$ is characteristic function of $\frac{X_t}{V_{X_t}}$.

Hence, there are such functions $(R_\nu)_{\nu=1}^\infty$ that we can write
\begin{equation*}
\Pro(X_t<x_2V_{X_t})-\Pro(X_t<x_1V_{X_t})=\Phi(x_2)-\Phi(x_1)+\sum_{\nu=1}^\infty \left (R_\nu(x_2)-R_\nu(x_1)\right)t^{-\frac{\nu}{2}}.
\end{equation*}
We use the classical Theorem~\ref{iid-case} and the scaling Lemma~\ref{scaling} and find out that for all $\nu=1,2,\dots$
\begin{equation*}
R_\nu(x)=Q_\nu^{X_1}(x)=t^{\frac{\nu}{2}}Q_{\nu}^{X_t}(x).
\end{equation*}
\end{proof}

\begin{proof}(Theorem~\ref{main2})

The proof proceeds analogously to the proof of Theorem~\ref{main} but we have to argue why we can rearrange the series of
\begin{equation}
\label{charfct}
f_{X_t}(s)=e^{-\frac{s^2}{2}}\exp{\left ( \sum_{j=1}^\infty \frac{\lambda_{j+2}^{X_1}}{(j+2)!}t^{-\frac{j}{2}}(is)^{j+2}\right )}.
\end{equation}
With present assumptions on the L\'{e}vy measure $\mu$, we can use the representation Lemma~\ref{cumulantlemma} for the cumulants. Let $m\in \N$ be such that $\frac{1}{m}\leq \epsilon$. Observe now that
\begin{align*}
&\int_0^\infty x^ne^{-x^{1+\frac{1}{m}}}dx=\int_0^\infty -\frac{m}{m+1}x^{n-\frac{1}{m}} \left ( -\frac{m+1}{m}x^{\frac{1}{m}}e^{-x^{1+\frac{1}{m}}}\right )dx\\
=& -\frac{m}{m+1}\left [ x^{n-\frac{1}{m}}e^{-x^{1+\frac{1}{m}}}\right ]_0^\infty +\int_0^\infty \frac{m}{m+1}\left ( n-\frac{1}{m}\right )x^{n-1-\frac{1}{m}}e^{-x^{1+\frac{1}{m}}}dx\\ =&\int_0^\infty \left ( \frac{m}{m+1}\right )^2\left ( n-\frac{1}{m}\right )\left ( n-1-\frac{2}{m}\right )x^{n-2-\frac{2}{m}}e^{-x^{1+\frac{1}{m}}}dx\\ =& \left ( \frac{m}{m+1}\right )^{\lfloor n \frac{m}{m+1}\rfloor}\prod_{j=1}^{\lfloor n \frac{m}{m+1}\rfloor}\left ( n+1-j\left ( 1+\frac{1}{m}\right )\right )\times \int_0^\infty x^{n-\lfloor n \frac{m}{m+1}\rfloor\left ( \frac{m+1}{m}\right )}e^{-x^{1+\frac{1}{m}}}dx\\\leq & \prod_{j=1}^{\lfloor n \frac{m}{m+1}\rfloor}\left ( n+1-j\left(1+\frac{1}{m}\right)\right )\times D,
\end{align*}
where
\begin{equation*}
D=\max_{l=0,\dots,m} \int_0^\infty x^{l-\lfloor l \frac{m}{m+1}\rfloor\left ( \frac{m+1}{m} \right)}e^{-x^{1+\frac{1}{m}}}dx.
\end{equation*}
Note that the constant $D$ is finite and does not depend on $n$. Without loss of generality, we can assume $\tilde{X}$ to be compensated compound Poisson process with $a=0$, since we can express general $X$ as a sum of this kind of process and a process satisfying the conditions of Theorem~\ref{main}. Then we get a bound for~(\ref{charfct}) by the additivity of cumulants.

Note that this decomposition can be made such a way that Cram\'{e}r's condition does not fail here if the L\'{e}vy measure has unbounded support. This is due to the fact that the tail of the L\'{e}vy measure is absolutely continuous with respect to Lebesgue measure. Now we have
\begin{align*}
&\sum_{\nu=2}^\infty \left | \frac{\gamma_\nu^{\tilde{X}_t}}{V_{X_t}^\nu\nu!}(is)^\nu\right |\\=&\sum_{\nu=2}^{m}\left | \frac{\gamma_\nu^{\tilde{X}_t}}{V_{X_t}^\nu\nu !} (is)^\nu\right |+t\sum_{\nu=m+1}^\infty \frac{1}{\nu !} \left|\gamma_\nu^{\tilde{X}_1}\right|\left ( \frac{|s|}{\sqrt{t}V_{X_1}}\right )^\nu \\ =& \sum_{\nu=2}^{m}\left | \frac{\gamma_\nu^{\tilde{X}_t}}{V_{X_t}^\nu \nu !}(is)^\nu\right |+t\sum_{j=1}^\infty \sum_{k=0}^{m} \frac{1}{((m+1)j+k)!}\left | \gamma_{(m+1)j+k}^{\tilde{X}_1}\right |\left ( \frac{|s|}{\sqrt{t}V_{X_1}}\right )^{(m+1)j+k}.
\end{align*}
The first term is a finite sum of finite summands if $0<t<\infty$. We get an estimate for the other sum as follows
\begin{align*}
&t\sum_{j=1}^\infty \sum_{k=0}^{m}\frac{1}{((m+1)j+k)!}\left | \gamma_{(m+1)j+k}^{\tilde{X}_1}\right |\left ( \frac{|s|}{\sqrt{t}V_{X_1}}\right )^{(m+1)j+k}\\ \leq & t \sum_{j=1}^\infty \sum_{k=0}^{m} \frac{1}{((m+1)j+k)!}\times \\&2CD \prod_{l=1}^{\lfloor ((m+1)j+k)\frac{m}{m+1}\rfloor}\left ( (m+1)j+k+1-l\left ( 1+\frac{1}{m}\right )\right )\left ( \frac{|s|}{\sqrt{t}V_{X_1}}\right )^{(m+1)j+k}.
\end{align*}
Now define
\begin{equation*}
g(l)=(m+1)j+k+1-l-\left \lfloor \frac{l}{m}\right \rfloor, \quad l=1,\dots,\left \lfloor((m+1)j+k)\frac{m}{m+1} \right \rfloor.
\end{equation*}
We observe that $g(l)>g(l+1)$ and the values of $g$ are integers from $1$ to $(m+1)j+k$. Nevertheless, $g$ does not take every $(m+1)$th integer value. This fact is due to the jump of the floor function. So there is at least $j$ terms missing in the product. By assuming them to be the $j$ smallest ones, we get a rough estimate
\begin{equation*}
\prod_{l=1}^{\lfloor ((m+1)j+k)\frac{m}{m+1}\rfloor}\left ( (m+1)j+k+1-l\left (1+\frac{1}{m} \right )\right )\leq \frac{((m+1)j+k)!}{j!}.
\end{equation*}
And finally
\begin{align*}
&t\sum_{j=1}^\infty \sum_{k=0}^{m}\frac{1}{((m+1)j+k)!}\left | \gamma_{(m+1)j+k}^{\tilde{X}_1}\right |\left ( \frac{|s|}{\sqrt{t}V_{X_1}}\right )^{(m+1)j+k}\\\leq& t \sum_{j=1}^\infty \sum_{k=0}^{m}\frac{1}{j!}2CD\left ( \frac{|s|}{\sqrt{t}V_{X_1}}\right )^{(m+1)j+k}\\\leq & 2CDt \left ( \sum_{k=0}^{m}\left ( \frac{|s|}{\sqrt{t}V_{X_1}}\right )^k\right )\sum_{j=1}^\infty \frac{1}{j!}\left ( \left ( \frac{|s|}{\sqrt{t}V_{X_1}}\right )^{m+1}\right )^j< \infty,
\end{align*}
as an exponential series when $0<t<\infty$. The last part of the proof is analogous to the proof of Theorem~\ref{main}.
\end{proof}

\begin{proof}(Theorem~\ref{main3})

We have to get a suitable estimate for the cumulants from above to be able to continue as in the proof of Theorem~\ref{main2}. Let us define function $\eta$ on positive reals as follows
\begin{equation*}
\eta (x)=\mu((-x-1,-x],[x,x+1)).
\end{equation*}
We can easily represent the growing condition for the L\'{e}vy measure using this function. Now we can estimate the cumulants in the spirit of Lemma~\ref{cumulantlemma}. Without loss of generality, we can assume that $a\geq 1$. We get
\begin{equation*}
\int_{-\infty}^\infty |x|^\nu \mu(dx)\leq \int_{-a}^a |x|^\nu \mu(dx)
+\sum_{j=0}^\infty |j+a+1|^\nu \eta(a+j).
\end{equation*}
For $\nu\geq 2$, the first term is bounded by $D^\nu$ for some $D>0$. For the second term we get
\begin{equation*}
\sum_{j=0}^\infty |j+a+1|^\nu \eta(a+j)\leq \int_a^{\infty} (x+2)^\nu C e^{-x^{1+\epsilon}}dx\leq C \int_a^\infty (3x)^\nu e^{-x^{1+\epsilon}}dx.
\end{equation*}
The rest of the proof is analogous to the proof of Theorem~\ref{main2}.
\end{proof}

\begin{proof}(Corollary~\ref{densitycase})

Let $(P_\nu)_{\nu=1}^\infty$ be the same polynomials as in the proof of Theorem~\ref{main}. We can use the series representation for characteristic function of $\frac{X_t}{V_{X_t}}$ and get
\begin{equation*}
\frac{1}{2\pi}\int_{\R} e^{-isx}f_{X_t}(s)ds=\frac{1}{2\pi}\int_{\R}e^{-isx}\left ( e^{-\frac{s^2}{2}}+\sum_{\nu=1}^\infty P_\nu(is)e^{-\frac{s^2}{2}}t^{-\frac{\nu}{2}}\right)ds.
\end{equation*}
With fixed $t>0$, the absolute convergence is uniform in compact sets with respect to $s$, as in the preceeding proofs. Thus, the integral is well-defined and can be computed term-wise. Moreover, with fixed $x\in \R$
\begin{equation*}
\int_{|s|>T}e^{-isx}f_{X_t}(s)ds\rightarrow 0, \quad \text{as }T\rightarrow \infty,
\end{equation*}
since $f_{X_t}$ is a characteristic function of some random variable with density function. Hence,
\begin{align*}
&\lim_{T \rightarrow \infty} \left | g_{X_t}(x)-\frac{1}{2\pi}\int_{-T}^T e^{-isx}\left ( e^{-\frac{s^2}{2}}+\sum_{\nu=1}^\infty P_\nu(is)e^{-\frac{s^2}{2}}t^{-\frac{\nu}{2}}\right )ds\right |\\=&\lim_{T\rightarrow \infty} \left|\frac{1}{2\pi}\int_{|s|>T}e^{-isx}f_{X_t}(s)ds \right|=0.
\end{align*}
We have shown that there is some series representation but we still have to show that the limit equals to what is claimed. We have
\begin{equation*}
\frac{1}{2\pi}\int_{\R}e^{-isx}P_\nu(is)e^{-\frac{s^2}{2}}t^{-\frac{\nu}{2}}ds=-\frac{1}{2\pi}\int_\R 
is 
\frac{e^{-isx}}{-is}P_\nu(is)e^{-\frac{s^2}{2}}t^{-\frac{\nu}{2}}ds=\frac{d}{dx}Q_\nu^{X_t}(x).
\end{equation*}
 \end{proof}

\section*{Acknowledgements}

I am grateful to my supervisor professor Esko Valkeila for his comments and guidance. My work has been funded by Finnish Academy of Science and Letters, Vilho, Yrj\"{o} and Kalle V\"{a}is\"{a}l\"{a} Foundation.







\bibliographystyle{elsart-harv}

\begin{thebibliography}{16}
\expandafter\ifx\csname natexlab\endcsname\relax\def\natexlab#1{#1}\fi
\expandafter\ifx\csname url\endcsname\relax
  \def\url#1{\texttt{#1}}\fi
\expandafter\ifx\csname urlprefix\endcsname\relax\def\urlprefix{URL }\fi

\bibitem[{Asmussen(2000)}]{asmussen}
Asmussen, S., 2000. Ruin Probabilities. Singapore: World Scientific.

\bibitem[{Asmussen and Rosi\'{n}ski(2001)}]{asmussenrosinski}
Asmussen, S., Rosi\'{n}ski, J., 2001. Approximations of small jumps of
  {L}\'{e}vy processes with a view towards simulation. J. Appl. Probab. 38,
  482--493.

\bibitem[{Beard et~al.(1977)Beard, Pentik\"{a}inen, and Pesonen}]{beard}
Beard, R.~E., Pentik\"{a}inen, T., Pesonen, E., 1977. Risk Theory The
  Stochastic Basis of Insurance. London: Chapman and Hall.

\bibitem[{Bertoin(1996)}]{bertoin}
Bertoin, J., 1996. L\'{e}vy Processes. Cambridge: Cambridge University Press.

\bibitem[{Cont and Tankov(2004)}]{cont}
Cont, R., Tankov, P., 2004. Financial Modelling With Jump Processes. Boca
  Raton: Chapman and Hall.

\bibitem[{Cram\'{e}r(1962)}]{cramer}
Cram\'{e}r, H., 1962. Random Variables and Probability Distributions.
  Cambridge: Cambridge University Press.

\bibitem[{Kl\"{u}ppelberg and Kyprianou(2006)}]{kluppelberg}
Kl\"{u}ppelberg, C., Kyprianou, A.~E., 2006. On extreme ruinous behaviour of
  {L}\'{e}vy insurance risk processes. J. Appl. Probab. 43~(2), 594--598.

\bibitem[{Kolassa(2006)}]{kolassa}
Kolassa, J.~E., 2006. Series Approximation Methods in Statistics. New York:
  Springer.

\bibitem[{Kyprianou(2006)}]{kyprianou}
Kyprianou, A.~E., 2006. Introductory Lectures on Fluctuations of L\'{e}vy
  Processes with Applications. Berlin: Springer.

\bibitem[{Nualart(1995)}]{nualart}
Nualart, D., 1995. The Malliavin Calculus and Related Topics. Springer, New
  York.

\bibitem[{Nualart and Schoutens(2000)}]{nualartschoutens}
Nualart, D., Schoutens, W., 2000. Chaotic and predictable representations for
  {L}\'{e}vy processes. Stochastic Process. Appl. 90, 109--122.

\bibitem[{Petrov(1975)}]{petrov2}
Petrov, V.~V., 1975. Sums of Independent Random Variables. Berlin:
  Springer-Verlag.

\bibitem[{Petrov(1995)}]{petrov1}
Petrov, V.~V., 1995. Limit Theorems of Probability Theory Sequences of
  Independent Random Variables. Oxford: Oxford Science Publications.

\bibitem[{Rudin(1987)}]{rudin}
Rudin, W., 1987. Real and Complex Analysis. Singapore: McGraw-Hill.

\bibitem[{Sato(1999)}]{sato}
Sato, K.-I., 1999. L\'{e}vy Processes and Infinitely Divisible Distributions.
  Cambridge: Cambridge University Press.

\bibitem[{Valkeila(1995)}]{valkeila}
Valkeila, E., 1995. On normal approximation of a process with independent
  increments. Russian Math. Surveys 50, 945--961.

\end{thebibliography}


\end{document}